\documentclass[11pt]{amsart}

\usepackage{amssymb}
\usepackage{mathrsfs}
\usepackage{amsfonts}
\usepackage{latexsym,amsmath,amsthm,amssymb,amsfonts}
\usepackage[usenames]{color}
\usepackage{xspace,colortbl}
\usepackage{graphicx}
\usepackage{tipa}

\newcommand{\be}{\begin{equation}}
\newcommand{\ee}{\end{equation}}
\newcommand{\beq}{\begin{eqnarray}}
\newcommand{\eeq}{\end{eqnarray}}

\newcommand{\uS}{\mathbb{S}^{n}}

\def\R{{\mathfrak R}}

\newtheorem{prop}{Proposition}[section]

\newtheorem{remark}[prop]{Remark}

\def\begeq{\begin{equation}}
\def\endeq{\end{equation}}

\def\R{\mathbb R}
\def\S{\mathbb S}

\def\odot{\setbox0=\hbox{$\bigcirc$}\relax \mathbin {\hbox
to0pt{\raise.5pt\hbox to\wd0{\hfil $\wedge$\hfil}\hss}\box0 }}

\numberwithin{equation} {section}

\numberwithin{equation}{section}
\newtheorem{theorem}{\bf Theorem}[section]
\newtheorem{proposition}[theorem]{\bf Proposition}

\newtheorem{lemma}[theorem]{\bf Lemma}

\allowdisplaybreaks

\begin{document}
\title[The $L_p$ dual Minkowski problem for capillary hypersurfaces]
 {The $L_p$ dual Minkowski problem for capillary hypersurfaces}

\author{ Ya Gao$^{\dagger,\ast}$}

\address{
 $^{\dagger}$School of Mathematical Science and Academy for Multidisciplinary Studies, Capital Normal University, Beijing 100048, China. }

\email{Echo-gaoya@outlook.com}

\thanks{$\ast$ Corresponding author}

\date{}
%\maketitle
\begin{abstract}
In this paper, we consider the $L_p$ dual Minkowski problem for capillary hypersurfaces for $p>q$ and $q\leq 1$, which aims to find a capillary convex body with a prescribed capillary $(p,q)$-th dual curvature measure in the Euclidean half-space. We reduce it to a Monge-Amp\`ere type equation with a Robin boundary condition on the unit spherical cap, we prove that there exists a unique smooth solution that solves this problem provided $\theta\in (0,\frac{\pi}{2})$.
\end{abstract}

\maketitle {\it \small{{\bf Keywords}: $L_p$ dual Minkowski problem, Capillary hypersurfaces, Monge-Amp\`ere equation, Robin boundary condition.}

}

\section{Introduction}
The classical Minkowski problem as a fundamental problem in the Brunn-Minkowski theory of convex geometry, aims to determine a convex body whose surface area measure corresponds to a given spherical Borel measure. This problem has been completely solved under a necessary and sufficient condition by Minkowski \cite{Min97, Min03}, Alexandrov \cite{Ale38, Ale39, Ale56}, Lewy \cite{Lew38}, Nirenberg \cite{Nir53}, Pogorelov \cite{Pog52}, Cheng-Yau \cite{CY76}, and others. 

The classical Brunn-Minkowski theory has been extended by the $L_p$ Brunn-Minkowski theory and the dual Brunn-Minkowski theory. Now, they are all becoming the center focus of convex geometry. The $L_p$-Minkowski problem which as a fundamental problem in the $L_p$ Brunn-Minkowski theory and greatly generalizes the classical Minkowski problem was introduced by Lutwak \cite{Lut.JDG.38-1993.131, Lut.Adv.118-1996.244} and has been extensively studied since then; see e.g.
\cite{%BHZ.IMRNI.2016.1807,
  BLYZ.JAMS.26-2013.831,
  CLZ.TAMS.371-2019.2623,
  Zhu.Adv.262-2014.909}
for the logarithmic Minkowski problem,
\cite{JLW.JFA.274-2018.826,
  JLZ.CVPDE.55-2016.41,
  Lu.SCM.61-2018.511,
  Lu.JDE.266-2019.4394,
  LW.JDE.254-2013.983,
  Zhu.JDG.101-2015.159}
for the centroaffine Minkowski problem, and
\cite{CW.Adv.205-2006.33,
  HLYZ.DCG.33-2005.699,
  % LYZ.JDG.56-2000.111,
  LYZ.TAMS.356-2004.4359,
  Sta.Adv.167-2002.160}
for other cases of the $L_p$-Minkowski problem. For another, the dual Minkowski problem was first proposed by Huang, Lutwak, Yang and Zhang
in their recent groundbreaking work \cite{HLYZ.Acta.216-2016.325} and then
followed by
\cite{BHP.JDG.109-2018.411,
  %CL.Adv.333-2018.87,
  HP.Adv.323-2018.114,
  HJ.JFA.277-2019.2209,
  %JW.JDE.263-2017.3230,
  LSW.JEMSJ.22-2020.893,
  Zha.CVPDE.56-2017.18,
  Zha.JDG.110-2018.543}.

  Recently, Lutwak, Yang and Zhang introduced the $L_p$ dual Minkowski problem in \cite{LYZ18}, which unifies the classical Minkowski problem, the $L_p$-Minkowski problem and the dual Minkowski problem. For the general $L_p$ dual Minkowski problem, much progress has already been made
\cite{BF.JDE.266-2019.7980,
  CHZ.MA.373-2019.953,
  CCL,
  % CL,
  HLYZ.JDG.110-2018.1,
  HZ.Adv.332-2018.57,
  LLL}. 
  In fact, for a convex body $K\subset \mathbb{R}^{n+1}$ containing the origin as its interior, the $L_p$ dual curvature measures $d\widetilde{C}_{p,q}$\footnote{Sometimes simply referred to as \emph{the $(p,q)$-th dual curvature measure}.} is defined as 
  \begin{equation*}\label{a12}
      d\widetilde{C}_{p,q} =\frac{1}{n+1}h_K^{1-p}\left(h_K^2 + |\nabla h_K|^2\right)^{\frac{q-n-1}{2}}\det(\nabla^2 h_K+h_K\delta_{ij})da,
  \end{equation*}
where $h_K$ is the support function of $K$, $\nabla h_K$ and $(h_K)_{ij}$ are the gradient and the Hessian of $h_K$ on the unit sphere $\uS$ with respect to an orthonormal basis respectively, 
$da$ is the standard spherical area measure.\\
{\bf The $L_p$ dual Minkowski problems.} \emph{Given a finite nonzero Borel measure $m$ on $\uS$ and real numbers $p,q$, does there exist a convex hypersurface (as a boundary of a convex body) such that its induced area measure $d\tilde{C}_{p,q}$ equals $m$?}

When the given measure $m$ has a density $f$, the $L_p$ dual Minkowski problems becomes the existence problem of the following Monge-Amp\`ere equation on $\uS$:
$$\det(h_{ij}+ h\delta_{ij} ) = fh^{p-1}(h^2 + |\nabla h|^2)^{\frac{n+1-q}{2}},$$
where $f$ is a given positive smooth function on $\uS$, $\delta_{ij}$ is the Kronecker delta, $\nabla h$ and $(h_{ij})$ are the gradient and the Hessian of $h$ on $\uS$ with respect to an orthonormal basis respectively.

From the above Monge-Amp\`ere equation, it can be seen that when $q=n+1$, the $L_p$ dual Minkowski problem becomes the $L_p$ Minkowski problem, which includes the classical Minkowski problem; when $p=0$, the $L_0$ dual Minkowski problems becomes the dual Minkowski problem, if further that $q=0$, the dual Minkowski problem becomes the Aleksandrov problem.

As classical Minkowski problems and its related problems have progressed, similar questions have naturally emerged with the boundary problems. Busemann \cite{Busemann59} considered the Minkowski's and related problems for convex surfaces with boundaries very early, then, Oliker \cite{Oli82} introduced and studied one boundary problem which equals to solve the Monge-Amp\`ere equation of classical Minkowski problem with the vanish Dirichlet boundary condition, some other Dirichlet boundary value problems have also been studied by \cite{CW95, Pog78, Sch18, Sch03}. Recently, some results introduce the problems with a Robin (or Neumann) boundary value condition for convex capillary hypersurfaces.

\subsection{Setup and the problem }

Let $\{E_i\}_{i=1}^{n+1}$ be the standard orthonormal basis of $\R^{n+1}$, $\R^{n+1}_{+}=\{x\in\R^{n+1}| x\cdot E_{n+1}>0\}$ be the upper Euclidean half-space, $\Sigma\subset\overline{\R^{n+1}_{+}}$ be a properly embedded, smooth compact hypersurface with boundary such that
$$int(\Sigma)\subset\R^{n+1}_{+}\qquad and \qquad\partial\Sigma\subset\partial\R^{n+1}_{+}.$$
We call $\Sigma\subset\overline{\R^{n+1}_{+}}$ a {\emph{capillary hypersurface}} if $\Sigma$ intersects $\partial\R^{n+1}_{+}$ at a constant contact angle $\theta\in (0,\pi)$. Let $\nu$ be the unit outward normal, i.e., the Gauss map of $\Sigma$ with respect to the domain $\hat{\Sigma}$, where $\hat{\Sigma}$ is a bounded closed region in $\overline{\R^{n+1}_{+}}$ enclosed by the convex capillary hypersurface $\Sigma$ and $\partial\R^{n+1}_{+}$, and denote $\hat{\partial\Sigma}:=\partial\hat{\Sigma}\backslash\Sigma\subset\partial\R^{n+1}_{+}$. The contact angle $\theta$ is defined by 
$$\cos(\pi-\theta) = \langle \nu , e\rangle,$$
where $e:= -E_{n+1}$, hence $e$ is the unit outward normal of $\partial\R^{n+1}_{+}$. If $\Sigma$ is convex, then the Gauss image $\nu(\Sigma)$ of $\Sigma$ lies in the spherical cap
$$\S^n_{\theta}:=\left\{x\in\uS | \langle x, E_{n+1}\rangle \geq \cos\theta\right\}.$$
Instead of the usual Gauss map $\nu$, we give another map
$$\tilde{\nu}:= T \circ \nu: \Sigma\to C_{\theta},$$
where $C_{\theta}$ is a spherical cap defined by
$$C_{\theta}:= \left\{\xi\in\overline{\R^{n+1}_{+}} |~ |\xi-\cos\theta\cdot e|=1\right\},$$
which also a capillary hypersurface, $T: \S^n_{\theta}\to C_{\theta}$ is a translation in the vertical direction define by $T(z)=z+\cos\theta\cdot e$. The diffeomorphism map $\tilde{\nu}$ called the {\emph{capillary Gauss map}} of $\Sigma$, thus we can reparametrize $\Sigma$ using its inverse on $C_{\theta}$ (we can refer to \cite[Fig. 1]{MWW25a}). Hence, we also view the usual support function $h$ of $\Sigma$ as a function defined on $C_{\theta}$. 
Denote $\mathcal{K}_{\theta}$ as the set of all capillary convex bodies in $\overline{\R^{n+1}_{+}}$, and $\mathcal{K}^{\circ}_{\theta}$ as the family of capillary convex bodies with the origin as an interior point of the flat part of their boundary.

Recently, Mei, Wang and Weng \cite{MWW25a} introduce a capillary Minkowski problem, which asks for the existence of a strictly convex capillary hypersurface $\Sigma\subset\overline{\R^{n+1}_{+}}$ with a prescribed Gauss-Kronecker curvature on a spherical cap $C_{\theta}$. Subsequently, they further considered a capillary $L_p$-Minkowski problem for $p\geq 1$ in \cite{MWW25b}. Otherwise, Hu, Ivaki and Scheuer \cite{HIS25} consider the capillary Christoffel-Minkowski problem, Hu and Ivaki \cite{HI25} solve the even capillary $L_p$-Minkowski problem for the range $-n< p < 1$ and $\theta\in (0, \frac{\pi}{2})$ using an iterative scheme. Wang and Zhu \cite{WZ25} introduce a more general case, i.e., the capillary Orlicz-Minkowski problem.

In addition, we also attempted to solve these problems using the curvature flow as a tool. From this point of view, Mei, Wang and Weng \cite{MWW25c} studied a Gauss curvature type flow for capillary hypersurface, i.e., capillary Gauss curvature flow. Hu, Hu and Ivaki \cite{HHI25} obtained the long-time existence and asymptotic behavior of a class of anisotropic capillary Gauss curvature flow, they also can obtain the existence of smooth solutions to the capillary even $L_p$ Minkowski problem in the Euclidean half-space and capillary $L_p$ Minkowski problem as applications for flow. 

In this paper, we want to study the {\emph{capillary $L_p$ dual Minkowski problem}}, i.e., study the capillary $(p,q)$-th dual curvature measure $d\tilde{C}^c_{p,q}$ for a convex capillary body in $\overline{\mathbb{R}_+^{n+1}}$, which is defined by
$$d\tilde{C}^c_{p,q}:= \frac{1}{n+1}lh^{1-p}\left(h^2 + |\nabla h|^2\right)^{\frac{q-n-1}{2}}\det(\nabla^2 h+h\sigma)d\sigma,$$
where $\nabla h$ and $\nabla^2 h$ are the gradient and the Hessian of $h$ on $C_{\theta}$ with respect to the standard spherical metric $\sigma$ on $C_{\theta}$ respectively, and $l:=\sin^2\theta + \cos\theta\langle\xi,e\rangle$, $\xi\in C_{\theta}$.\\ 
\\
{\bf Capillary $L_p$ dual Minkowski problem.} \emph{Given a positive smooth function $f$ on $C_{\theta}$, does there exist a capillary convex body $\Sigma\in \mathcal{K}^{\circ}_{\theta}$ such that its capillary $(p,q)$-th dual curvature measure $d\tilde{C}^c_{p,q}$ equals to $fld\sigma$ ?}

By applying a similar argument to that in \cite[Proposition 2.4]{MWW25a}, the capillary $L_p$ dual Minkowski problem is actually equivalent to solve the following Neumann boundary value problem of the Monge-Amp\`ere type equation:
\begin{equation}\label{Eq}
\left\{
\begin{aligned}
&\det\left(\nabla^2 h + h\sigma\right)=fh^{p-1}\left(h^2 + |\nabla h|^2\right)^{\frac{n+1-q}{2}}, ~~&&in~~C_{\theta}\\
&\nabla_{\mu} h=\cot\theta\cdot h,  ~~&& on~~\partial C_{\theta}
\end{aligned}
\right.
\end{equation}
where $\mu$ is the unit outward normal of $\partial C_{\theta}\subset C_{\theta}$. Here, if 
$A:=\nabla^2 h + h\sigma >0 ~~in~C_{\theta}$
is a solution to Eq. \eqref{Eq}, then we called it by the convex solution.

\subsection{Main results}

We consider the capillary $L_p$ dual Minkowski problem for $p>q$, as formulated in Eq. \eqref{Eq}. The case $q=n+1$ is the capillary $L_p$-Minkowski problem, if $p>q=n+1$, it have been solved in \cite{MWW25b}. The main theorem of this paper extends the previous conclusion to $L_p$ dual Minkowski problem when $q\leq 1$ and it stated below.

\begin{theorem}\label{main1.1}
Let $p>q$, $q\leq 1$ and $\theta\in (0, \frac{\pi}{2})$. For any positive smooth function $f$ defined on $C_{\theta}$, then there exists a unique smooth solution $h$ solving Eq. \eqref{Eq}. 
\end{theorem}
In fact, we also can get the conclusion of $L_p$-Minkowski problem, i.e., the case of $p>q=n+1$ but we omit it in here. 

To prove the main Theorem,  we need to establish a priori estimate for solution to Eq. \eqref{Eq} up to the second derivative, and then apply the continuity method to obtain the existence of solution. First of all, we provide the uniform positive lower and upper bounds for solution to prevent the Eq. \eqref{Eq} to degenerate. 

When $p>q$, we use the capillary support function $u(\xi)$ which satisfies a Monge-Amp\`ere type equation \eqref{a2} with vanish Neumann boundary condition, by applying the maximum principle to establish the uniform positive lower and upper bounds for the solutions. Since the uniform bounded affected by the signs of $(p-1)$ and $(n+1-q)$, we analysis it by different cases, specific conclusions can be found in the Lemma \ref{C0 estimate}. Then, we give the roughly $C^1$-estimate which is based solely on convexity and is independent of the specific equation in Lemma \ref{c1-1}, i.e., the gradient estimate of solution $h$ can be bounded by constant contact angle $\theta$ and upper bound of $h$ as follow:
$$\max_{C_{\theta}}|\nabla h|\leq (1+\cot^2\theta)^{1/2}\Vert h\Vert_{C^0 (C_{\theta})},$$
so we have the upper bound of $|\nabla h|$. 

For the $C^2$ estimate, unlike approach in \cite{MWW25b}: they adopt the approach initiated by Lions-Trudinger-Urbas \cite{LTU86}, first reduce the global $C^2$ estimate to the boundary double normal estimate by choosing a suitable test function and then deduce the boundary double normal $C^2$ estimate by constructing an another suitable test function. Based on the general form of Eq. \eqref{Eq} and inspired by \cite{HIS25}, we have chosen a new auxiliary function
$$P:= \sigma_1 + \frac{1}{2}|\nabla h|^2 + Mh$$
to obtain the $C^2$ estimate, where $\sigma_1$ denote the trace of matrix $(A_{ij}): = (h_{ij}+ h\delta_{ij})$, $M$ is negative constant to be chosen later. Similarly, we first prove that the global $C^2$ estimate can be controlled by a constant or the boundary double normal estimate. Then, under the same auxiliary function without choosing a new text function, we prove the boundary double normal $C^2$ estimate.

\begin{remark}
    {\bf (a)} Although we obtain the $C^2$ estimate unlike approach in \cite{MWW25b}, we also can get the conclusion of $L_p$-Minkowski problem, i.e., the case of $p>q=n+1$ and we omit it in Theorem \ref{main1.1};\\
    {\bf (b)}
    We conclude that the range $\theta\in (0,\frac{\pi}{2})$ essential for obtaining the global $C^2$ estimate, as it ensures the strict convexity of $\partial C_{\theta}\subset C_{\theta}$. However, this restriction on the range is not necessary in establishing the $C^0$ and $C^1$ estimates for Eq. \eqref{Eq};\\ 
    {\bf (c)}
    We conclude that $q\leq 1$ essential for obtaining the interior $C^2$ estimate, as it ensures that the coefficient of the highest degree term of $\sigma_1$ is positive. However, we suspect that the condition $q\leq 1$ in Theorem \ref{main1.1} can be omitted.
\end{remark}

This paper is organized as follows. In Section 2, we will show some preliminaries which play a crucial role in our paper. In Section 3 and Section 4, several estimates, including $C^0$ estimate, gradient estimate and curvature estimate, of solutions to the Eq. \eqref{Eq} will be shown in details. In Section 5, we obtain the uniqueness and use the continuity method to prove Theorem \ref{main1.1}.

\section{Perliminaries } \label{se2}
This section is divided into two subsections. In the first subsection, we introduce the notion of capillary $L_p$ dual curvature measure, i.e., the capillary $(p,q)$-th dual curvature measure. 
In the second subsection, we provide some basic properties of capillary convex hypersurfaces.  

\subsection{The capillary $L_p$ dual Minkowski problem}

In this subsection, we introduce the $L_p$ dual curvature measures in \cite{LYZ18} and its situation under our capillary setting. 

Let $\Sigma$ be a capillary hypersurface of $\overline{\mathbb{R}^{n+1}_+}$, the support function $h$ of $\Sigma$ is defined by 
$$h (y) = \max \{ x\cdot y : x\in \Sigma\}, \quad y\in\mathbb{R}^{n+1}.$$

Suppose $\Sigma\in \mathcal{K}_{\theta}^{\circ}$, the radial function $\rho$ and radial map $r$ defined by
\begin{equation*}
    \rho (x) = \max\{\lambda: \lambda x\in K\}, \qquad x\in \mathbb{R}^{n+1}\backslash\{0\}
\end{equation*}
and
\begin{equation*}
    r: \uS_+\to \Sigma, \qquad r(\gamma)=\rho(\gamma)\gamma \in\Sigma
\end{equation*}
respectively, where $\uS_+$ is the upper sphere. For $\gamma\in\uS_+$, define the \emph{radial Gauss map $\alpha:\uS_+\to \uS_{\theta}$ of $\gamma$} by 
\begin{equation*}
    \alpha(\gamma) = \nu(\rho (\gamma)\gamma)\subset \uS_{\theta}
\end{equation*}
and $\alpha^{\ast}: \uS_{\theta} \to \uS_+$ is its reverse map. Similar to the definition of capillary Gauss map, instead of the usual radial Gauss map $\alpha$, we give another map 
\begin{equation*}
    \widetilde{\alpha}:= T\circ \alpha \circ r^{-1}  : \Sigma \to C_{\theta},
\end{equation*}
so its reverse map is $\widetilde{\alpha^{\ast}}: C_{\theta} \to \Sigma$. The map $\widetilde{\alpha}$ called the \emph{capillary radial Gauss map} of $\Sigma$, so we view the usual radial function $\rho$ of $\Sigma$ as a function defined on $C_{\theta}$.

For $K,L\in\mathcal{K}^{\circ}_{\theta}$ and $\lambda, \tau\geq 0$, the \emph{Minkowski combination} $\lambda K+ \tau L$ is defined by $\lambda K+ \tau L :=\{\lambda x+ \tau y: x\in K, y\in L\}$ and 
$$h_{\lambda K+ \tau L}:= \lambda h_{K} + \tau h_{L},$$
where $h_K$ and $h_L$ denote the support functions of the convex body $K$ and $L$ respectively.

The \emph{$L_p$ combination}, an extension of Minkowski combinations studied by Firey in the early 1960's, as defined by 
\begin{equation*}
    h^p_{\lambda K+_{p}\tau L }:= \lambda h_{K}^p + \tau h_{L}^p
\end{equation*}
for each $p\geq 1$, $K, L\in\mathcal{K}^{\circ}_{\theta}$ and $\lambda, \tau \geq 0$.

The $(p,q)$-th dual curvature measure was introduced by Lutwak, Yang and Zhang in \cite{LYZ18}. Based on the definition in \cite{LYZ18} and the above introduction, we can define the capillary $(p,q)$-th dual curvature measure of $\Sigma\in \mathcal{K}_{\theta}^{\circ}$ as 
$$ d\widetilde{C}_{p,q}^c =\frac{1}{n+1}lh_{\Sigma}^{1-p}\left(h_{\Sigma}^2 + |\nabla h_{\Sigma}|^2\right)^{\frac{q-n-1}{2}}\det(\nabla^2 h_{\Sigma}+h_{\Sigma}\sigma)d\sigma.$$

As mentioned earlier, here we have also give a more general area measure for the capillary convex body $\Sigma\in \mathcal{K}^{\circ}_{\theta}$. In particular, we can find that if $q=n+1$, it coincides with the capillary $L_p$-surface area measure, our conclusion will be reduced to the specific case that $p>n+1$ in \cite{MWW25b}.

\subsection{Basic properties of capillary convex hypersurfaces}
We first recall some basic properties of convex hypersurfaces in $\mathbb{R}^{n+1}$. As the set out in the previous section, let $\Sigma\subset\overline{\mathbb{R}^{n+1}_{+}}$ be a smooth, properly embedded, strictly convex capillary hypersurface, we parametrize $\Sigma$ by the inverse capillary Gauss map, i.e., $X: C_{\theta}\to \Sigma$ given by 
$$X(\xi) = \tilde{\nu}^{-1}(\xi) = \nu^{-1}\circ T^{-1}(\xi) = \nu^{-1}(\xi-\cos\theta e), \quad \xi\in C_{\theta}.$$
The usual support function $h$ of $\Sigma$ is given by
$$h(X) : =\langle X, \nu(X)\rangle.$$
The support function $h$ is equivalently defined by the unique decomposition
$$X = h(X)\nu (X) + s, \qquad s\in T_{X}\Sigma.$$
Now, we define the capillary support function $u$ by the following unique decomposition
\begin{equation}\label{e1}
     X = u(X)\tilde{\nu}(X) + s', \qquad s'\in T_{X}\Sigma.
\end{equation}
It's easy to see that
$$h(X) = u(X)\langle\nu(X), \tilde{\nu}(X)\rangle = u(X)(1+\cos\theta\langle\nu,e\rangle)$$
since $\tilde{\nu} := T\circ \nu = \nu + \cos\theta e$.
Now, let 
$$h(\xi):= \langle X(\xi),\nu(X(\xi))\rangle = \langle X(\xi), T^{-1}(\xi)\rangle = \langle\tilde{\nu}^{-1}(\xi), \xi-\cos\theta e\rangle,$$
referring to \cite[Proposition 2.4]{MWW25a}, it satisfies 
$$\nabla_{\mu}h = \cot\theta h, \qquad on~ \partial C_{\theta},$$
and 
$$u(\xi) = \frac{h(\xi)}{|\xi|^2 - \cos\theta\langle \xi,e\rangle} = \frac{h(\xi)}{\sin^2\theta + \cos\theta\langle \xi,e\rangle} = \frac{h(\xi)}{l(\xi)},$$
due to $|\xi - \cos\theta e|^2 =1$.

Now we compute the geometric quantities of $X$ in terms of $h$ and $u$. Let $\{e_i\}_{i=1}^n$ be a local orthonormal frame on $C_{\theta}$ such that along the $\partial C_{\theta}$, $e_n = \mu$ is the unit outward normal of $\partial C_{\theta}\subset C_{\theta}$. Together with $T^{-1}(\xi) = \xi -\cos\theta e$, which is the usual normal of $C_{\theta}$ at $\xi$, it builds on a local orthonormal frame in $\R^{n+1}_{+}$. Hence by the definition of $h$, for $X(\xi)\in\Sigma$,
$$X = \sum_{i=1}^n\langle X, e_i\rangle e_i + \langle X, T^{-1}(\xi)\rangle T^{-1}(\xi) = \sum_{i=1}^n\langle X, e_i\rangle e_i + h(\xi)T^{-1}(\xi).$$
Since $\nabla$ is the standard connection in $C_{\theta}$ with respect to the standard spherical metric $\sigma$, it's easy to see
$$\sum_{i=1}^n \langle X, e_i\rangle e_i = \nabla h(\xi).$$
Hence
$$X(\xi) = \nabla h(\xi) + h(\xi)T^{-1}(\xi).$$
By the direct computation
$$\nabla_{e_j}X = (\nabla_{ij}h + h\delta_{ij})e_i,$$
where $\delta_{ij}$ is the coefficients of the standard metric $\sigma$ on $\uS$. Hence, the second fundamental form $A_{ij}$ of $\Sigma$ is given by
$$A_{ij} = \langle \nabla_{e_i}X, \nabla_{e_i}(\xi-\cos\theta e)\rangle = \langle \nabla_{e_j}X, e_i\rangle = \nabla_{ij}h + h\delta_{ij},$$
and the induced metric $g_{ij}$ of $\Sigma\subset\overline{\R^{n+1}_{+}}$ can be derived by Weingarten's formula
$$g_{ij} = \langle \nabla_{e_i}X, \nabla_{e_j}X\rangle = A_{ik}A_{jl}\delta^{kl}.$$
The Gauss-Kronecker curvature of $\Sigma$ at $X$ is 
$$K_{G}(X(\xi)) = \det (g^{ik}A_{kj}) = \det (\nabla ^2 h(\xi)+ h(\xi)\sigma)^{-1}.$$

Both functions and formula will play a crucial role in our paper. 
Next, we give the equation of $h$ and $u$ on $\partial C_{\theta}$, which will be used in subsequent proof process.

Along the boundary $\partial C_{\theta}$, we choose an orthonormal frame $\{e_{i}\}_{i=1}^n$ with $e_n = \mu$ and $\mu$ is the unit outer normal of $\partial C_{\theta}$. Then 

\begin{proposition}\label{pro}
    The support function $h$ and the capillary support funtion $u=l^{-1}h$ satisfy the following boundary conditions on $\partial C_{\theta}$:
    \begin{equation}\label{pro1}
        h_{kn}=0,
    \end{equation}
    which is equivalent to 
    \begin{equation}\label{pro2}
        u_{kn} = -\cot\theta u_{k},
    \end{equation}
    where $k=1,2,\cdots, n-1$.
\end{proposition}

\begin{proof}
    First, taking tangential derivative of the equation $h_n = \cot\theta h$ along $\partial C_{\theta}$, we have $\nabla_{e_k}h_n = \cot\theta h_k$. It follows that 
    \begin{equation*}
\begin{split}
    h_{kn} &= \nabla^2 h(e_k,e_n) = \langle \nabla_{e_k}(\nabla h), e_n\rangle  \\ 
    &= 
    \nabla_{e_k}(\langle \nabla h, e_n\rangle) - \langle\nabla h,\nabla _{e_k}e_n\rangle \\ 
    &= \nabla_{e_k}h_n -\langle \nabla h, \cot\theta e_k\rangle \\ 
    &= 
    \cot\theta h_k - \cot\theta h_k = 0,
\end{split}
    \end{equation*}
    where we have used $\nabla_{e_k}e_n = \cot\theta e_k$. Then the equation \eqref{pro1} follows. Using the same computation and the fact $u_n =0$ in \eqref{a2}, we also have \eqref{pro2}. We complete the proof.
\end{proof}

\section{A priori estimates} \label{se3}
In this section, we give a priori estimates (including $C^0$ and $C^1$ estimates) for solutions to Eq. \eqref{Eq}.

\begin{lemma}[{\bf $C^0$ estimate}]\label{C0 estimate}
    Let $p>q$ and $\theta\in(0,\pi)$. Suppose $h$ is a positive solution to Eq. \eqref{Eq}. Then there holds
    \begin{equation}\label{a1}
        2^{-\frac{n+1-q}{2}}\frac{(1-\cos\theta)^{p-q}}{(\sin\theta)^{2(p-1)}}\min_{C_{\theta}}f^{-1}\leq h^{p-q} \leq 
        \frac{(\sin\theta)^{2(p-q)}}{(1-\cos\theta)^{n+p-q}}\max_{C_{\theta}}f^{-1}
    \end{equation}
    while $n+1-q\geq 0$ and $1-p\leq 0$, and
      \begin{equation}\label{a11}
        2^{-\frac{n+1-q}{2}}\frac{1}{(1-\cos\theta)^{q-1}}\min_{C_{\theta}}f^{-1}\leq h^{p-q} \leq 
        \frac{(\sin\theta)^{2(p-q)}}{(1-\cos\theta)^{n+p-q}}\max_{C_{\theta}}f^{-1}
    \end{equation}
    while $n+1-q\geq 0$ and $1-p>0$.
    \begin{equation}\label{a1-1}
        \frac{(1-\cos\theta)^{p-q}}{(\sin\theta)^{2(n+p-q)}}\min_{C_{\theta}}f^{-1}\leq h^{p-q} \leq 
        2^{-\frac{n+1-q}{2}}\frac{(\sin\theta)^{2(p-q)}}{(1-\cos\theta)^{p-1}}\max_{C_{\theta}}f^{-1}
    \end{equation}
    while $n+1-q<0$ and $1-p\leq 0$, and 
     \begin{equation}\label{a11}
        \frac{(1-\cos\theta)^{p-q}}{(\sin\theta)^{2(n+p-q)}}\min_{C_{\theta}}f^{-1}\leq h^{p-q} \leq 
        2^{-\frac{n+1-q}{2}}\frac{1}{(\sin\theta)^{2(q-1)}}\max_{C_{\theta}}f^{-1}
    \end{equation}
    while $n+1-q<0$ and $1-p>0$.
\end{lemma}

\begin{proof}
    Consider the capillary support function (see, e.g., \cite[Eq. (1.5)]{MWW25a})
    \begin{equation*}
        u:=l^{-1}h.
    \end{equation*}
Since $h$ is a solution to Eq.\eqref{Eq}, we know that $u$ satisfies
\begin{equation}\label{a2}
\left\{
\begin{aligned}
&\det\left(l\nabla^2 u + \cos\theta\cdot(\nabla u\otimes e^{T} + e^{T}\otimes \nabla u)+ u\sigma\right) \\ 
&\qquad\qquad\qquad\qquad=f(ul)^{p-1}\left((ul)^2 + |\nabla (ul)|^2\right)^{\frac{n+1-q}{2}}, ~~&&in~C_{\theta}\\
&\nabla_{\mu} u=0,  ~~&& on~\partial C_{\theta}
\end{aligned}
\right.
\end{equation}
where $e^{T}$ is the tangential part of $e$ on $C_{\theta}$, and we used that $\nabla_{\mu}l = \cot\theta\cdot l$ on $\partial C_{\theta}$.
We divide the proof into two cases: either $n+1-q\geq 0$ or $n+1-q<0$.

{\bf Case 1: } $n+1-q\geq 0$. Suppose $u$ attains the maximum value at some point $\xi_{0}\in C_{\theta}$. If $\xi_0 \in C_{\theta}\backslash\partial C_{\theta}$, then
\begin{equation}\label{a3}
    \nabla u(\xi_{0}) = 0 \qquad and \qquad \nabla^2 u(\xi_{0})\leq 0.
\end{equation}
If $\xi_0 \in \partial C_{\theta}$, the boundary condition in \eqref{a2} implies that \eqref{a3} still holds. Substituting \eqref{a3} into \eqref{a2} at $\xi_0$, we obtain
$$f(ul)^{n+p-q}\leq \det(l\nabla^2 u + \cos\theta(\nabla u \otimes e^{T} + e^{T}\otimes\nabla u)+u\sigma) \leq u^n,$$
which implies
$$u^{p-q}(\xi_0)\leq f^{-1}l^{q-n-p}\leq \frac{1}{\min_{C_{\theta}}f\cdot (1-\cos\theta)^{n+p-q}}.$$
Therefore, for all $\xi\in C_{\theta}$, by $(\sin\theta)^2\geq l\geq 1-\cos\theta$, we have
\begin{equation}\label{a4}
    \begin{split}
        h^{p-q}(\xi) &= (ul)^{p-q}(\xi) \\ 
        &\leq u^{p-q}(\xi_{0})\cdot\left(\max_{\xi\in C_{\theta}}l(\xi)\right)^{p-q} \\ 
        &\leq 
        \frac{(\sin\theta)^{2(p-q)}}{\min_{C_{\theta}}f\cdot (1-\cos\theta)^{n+p-q}}.
    \end{split}
\end{equation}
Similarly, if $u$ attains the minimum value at some point $\xi_1\in C_{\theta}$, there holds
\begin{equation}\label{a5}
        u^{p-q}(\xi_1) \geq 2^{-\frac{n+1-q}{2}}f^{-1}l^{1-p}.
\end{equation}
When $1-p\leq 0$, then we have
\begin{equation*}
    u^{p-q}(\xi_1) \geq
        2^{-\frac{n+1-q}{2}}\frac{1}{\max_{C_{\theta}}f\cdot (\sin\theta)^{2(p-1)}},
\end{equation*}
when $1-p>0$, then we have
\begin{equation*}
    u^{p-q}(\xi_1) \geq 2^{-\frac{n+1-q}{2}}\frac{1}{\max_{C_{\theta}}f\cdot (1-\cos\theta)^{p-1}}.
\end{equation*}
Therefore, for all $\xi\in C_{\theta}$, if $1-p\leq 0$, we have
\begin{equation}\label{a6}
    \begin{split}
        h^{p-q}(\xi) &= (ul)^{p-q}(\xi) \\ 
        &\geq 
        u^{p-q}(\xi_1)\left(\min_{\xi\in C_{\theta}}l(\xi)\right)^{p-q} \\ 
        &\geq 
        2^{-\frac{n+1-q}{2}}\frac{(1-\cos\theta)^{p-q}}{\max_{C_{\theta}}f\cdot (\sin\theta)^{2(p-1)}},
    \end{split}
\end{equation}
if $1-p>0$, we have
\begin{equation}\label{a9}
    \begin{split}
        h^{p-q}(\xi) &= (ul)^{p-q}(\xi) \\ 
        &\geq 
        u^{p-q}(\xi_1)\left(\min_{\xi\in C_{\theta}}l(\xi)\right)^{p-q} \\ 
        &\geq 
        2^{-\frac{n+1-q}{2}}\frac{1}{\max_{C_{\theta}}f\cdot (1-\cos\theta)^{q-1}}.
    \end{split}
\end{equation}

{\bf Case 2: } $n+1-q< 0$. Similarly, suppose $u$ attains the maximum value at some point $\xi_{0}\in C_{\theta}$. At $\xi_0$, we obtain
$$2^{\frac{n+1-q}{2}}fl^{p-1}u^{n+p-q}\leq \det(l\nabla^2 u + \cos\theta(\nabla u \otimes e^{T} + e^{T}\otimes\nabla u)+u\sigma) \leq u^n,$$
which implies
$$u^{p-q}(\xi_0)\leq 2^{-\frac{n+1-q}{2}}f^{-1}l^{1-p}.$$
When $1-p\leq 0$, we have
\begin{equation*}
    u^{p-q}(\xi_0)\leq 2^{-\frac{n+1-q}{2}}\frac{1}{\min_{C_{\theta}}f\cdot (1-\cos\theta)^{p-1}},
\end{equation*}
when $1-p>0$, we have
\begin{equation*}
    u^{p-q}(\xi_0)\leq 2^{-\frac{n+1-q}{2}} \frac{1}{\min_{C_{\theta}}f \cdot (\sin\theta)^{2(p-1)}}.
\end{equation*}
Therefore, for all $\xi\in C_{\theta}$, by $(\sin\theta)^2\geq l\geq 1-\cos\theta$, if $1-p\leq 0$, we have
\begin{equation}\label{a7}
    \begin{split}
        h^{p-q}(\xi) &= (ul)^{p-q}(\xi) \\ 
        &\leq u^{p-q}(\xi_{0})\cdot\left(\max_{\xi\in C_{\theta}}l(\xi)\right)^{p-q} \\ 
        &\leq 
        2^{-\frac{n+1-q}{2}}\frac{(\sin\theta)^{2(p-q)}}{\min_{C_{\theta}}f\cdot (1-\cos\theta)^{p-1}},
    \end{split}
\end{equation}
if $1-p>0$, we have
\begin{equation}\label{a10}
    \begin{split}
        h^{p-q}(\xi) &= (ul)^{p-q}(\xi) \\ 
        &\leq u^{p-q}(\xi_{0})\cdot\left(\max_{\xi\in C_{\theta}}l(\xi)\right)^{p-q} \\ 
        &\leq 
        2^{-\frac{n+1-q}{2}}\frac{1}{\min_{C_{\theta}}f\cdot (\sin\theta)^{2(q-1)}}.
    \end{split}
\end{equation}
Similarly, then there holds
\begin{equation}\label{a8}
        h^{p-q}(\xi)\geq \frac{(1-\cos\theta)^{p-q}}{\max_{C_{\theta}}f\cdot (\sin\theta)^{2(n+p-q)}}.
\end{equation}
This completes the proof.
\end{proof}

Next, we give the $C^1$-estimate, which is based solely on convexity and is independent of the specific equation.

\begin{lemma}[{\bf $C^1$ estimate}]\label{c1-1}
    Let $p>q$ and $\theta\in (0,\pi)$. Suppose $h$ is a positive solution to Eq. \eqref{Eq}, then there holds
    \begin{equation}\label{b1}
        \max_{C_{\theta}}|\nabla h| \leq (1+\cot^2\theta)^{1/2} \Vert h\Vert_{C^{0}(C_{\theta})}.
    \end{equation}
\end{lemma}

\begin{proof}
    We consider the function
    $$P:= |\nabla h|^2 + h^2.$$
    Suppose that the function $P$  attains its maximum value at some point $\xi_{0}\in C_{\theta}$. If $\xi_0 \in C_{\theta}\backslash\partial C_{\theta}$, we have
    $$0=\nabla_{e_i}P = 2h_{k}h_{ki} + 2hh_{i}, \quad for \quad 1\leq i\leq n.$$
Together with the convexity of $h$, i.e., $h_{ij}+h\delta_{ij}>0$, it follows $\nabla h(\xi_{0})=0$, so \eqref{b1} holds.

If $\xi_{0}\in \partial C_{\theta}$, we choose an orthonormal frame $\{e_i\}_{i=1}^{n}$ around $\xi_{0}\in\partial C_{\theta}$ such that $e_n = \mu$. From Proposition \ref{pro}, we know that
\begin{equation}\label{b2}
    h_{kn}=0 \qquad for~~any \quad 1\leq k\leq n-1.
\end{equation}
Then we have
$$0=\nabla_{e_k}P = 2\sum_{i=1}^{n}h_{i}h_{ik}+ 2hh_{k},$$
which implies
\begin{equation}\label{b3}
    h_{k}(\xi_{0}) = 0.
\end{equation}
Together with \eqref{b3} and $h_n = \cot\theta\cdot h$ on $\partial C_{\theta}$, we have
$$|\nabla h|^2(\xi_{0})\leq (|\nabla h|^2 + h^2)(\xi_0) = (h_n^2 + h^2)(\xi_{0}) = (1+\cot^2\theta)h^2(\xi_0).$$
In conclusion, \eqref{b1} holds.
\end{proof}

\section{$C^2$ estimates} \label{se4}
Now, we establish the priori $C^2$ estimate for the positive solution of Eq. \eqref{Eq}. 

\begin{lemma}[{\bf $C^2$ estimate}]\label{c2 estimate}
    Let $p>q$, $q=n+1$ or $q\leq 1$ and $\theta\in (0,\frac{\pi}{2})$. Suppose that $h$ is a positive solution to Eq. \eqref{Eq}, then there holds 
    \begin{equation}\label{d1}
        \max_{C_{\theta}}|\nabla^2 h| \leq C,
    \end{equation}
    where the positive constant $C$ depends on $n$, $p$, $q$, $\min_{C_{\theta}} f$, $\min_{C_{\theta}}h$, $\Vert f\Vert_{C^2 (C_{\theta})}$ and $\Vert h\Vert_{C^0 (C_{\theta})}$.
\end{lemma}

\begin{proof}
We consider the function 
    \begin{equation*}
        P(\xi):=\sigma_1 + \frac{1}{2}|\nabla h|^2 + M h
    \end{equation*}
    for $\xi\in C_{\theta}$, where $\sigma_1$ denote the trace of matrix $(A_{ij}):= (h_{ij} + h\delta_{ij})$, $M$ is negative constant to be chosen later. Suppose that $P$ attains its maximum value at some point $\xi_{0}\in C_{\theta}$. We divide this proof into two cases: either $\xi_0\in \partial C_{\theta}$ or $\xi_0 \in C_{\theta}\backslash\partial C_{\theta}$.

    {\bf Case 1.} $\xi_0 \in C_{\theta}\backslash\partial C_{\theta}$. In this case, we choose an orthonormal frame $\{e_i\}_{i=1}^n$ around $\xi_0$, such that $A_{ij} = (h_{ij} + h\delta_{ij})$ is diagonal, hence $\frac{\partial F(A)}{\partial A_{ij}}$ and $h_{ij}$ are also diagonal at $\xi_0$. Denote
    \begin{equation}\label{d3}
        F(A) := \log\det (A) =\log\left( h^{p-1}(h^2 + |\nabla h|^2)^{\frac{n+1-q}{2}}f\right) = :\tilde{f},
    \end{equation}
    and 
    $$F^{ij}:= \frac{\partial F(A)}{\partial A_{ij}}, ~~ F^{ij,kl}:=\frac{\partial^2 F(A)}{\partial A_{ij}\partial A_{kl}}.$$
    Then we have 
    \begin{equation}\label{d4}   
        F^{ij} = A^{ij},
    \end{equation}
    \begin{equation}\label{d5}
        F^{ij,kl} = -F^{ik}F^{jl} = - A^{ik} A^{jl},
    \end{equation}
    where $A^{ij}$ is the inverse of $A_{ij}$.
Note that
\begin{equation}\label{d6}
    F^{ij}A_{ij} =  n .
\end{equation}
Using the fact
$$h_{klij} = h_{ijkl} + 2h_{kl}\delta_{ij} - 2h_{ij}\delta_{kl} + h_{li}\delta_{kj} - h_{kj}\delta_{il},$$
then we have 
\begin{equation}\label{d7}
    \begin{split}
        F^{ij}(\sigma_1)_{ij} &= \sum_k F^{ij}A_{kkij} = \sum_k F^{ij}(h_{kkij}+ h_{ij}) \\ 
        &=
          \sum_k F^{ij}(A_{ijkk} + h_{kk}\delta_{ij} - h_{ij})\\ 
        &= -\sum_k F^{ij,mn}A_{ijk}A_{mnk} + \Delta \tilde{f} + \sigma_1\sum_i F^{ii} -n \\
        &=
        \sum_k F^{im}F^{jn}A_{ijk}A_{mnk}+\Delta \tilde{f} + \sigma_1\sum_i F^{ii} -n \\
        &= 
        \sum_k F^{ii}F^{jj}A_{ijk}^2 + \Delta \tilde{f} + \sigma_1\sum_i F^{ii} -n\\
        &\geq \frac{1}{n} \sum_k (\tilde{f})^2_k + \Delta\tilde{f} + \sigma_1 \sum_i F^{ii} - n.
    \end{split}
\end{equation}
By the directly computation, we know that
\begin{equation}\label{d8}
    \tilde{f} = (p-1)\log h + \frac{n+1-q}{2}\log (h^2 + |\nabla h|^2) + \log f,
\end{equation}

\begin{equation}\label{d9}
   \tilde{f}_i = (p-1)\frac{h_i}{h} + (n+1-q)\frac{h_i A_{ii}}{h^2 + |\nabla h|^2} + (\log f)_i,
\end{equation}
so
\begin{equation}\label{d10}
\begin{split}
    (\tilde{f}_i)^2 &= (n+1-q)^2 \frac{h_i^2 A_{ii}^2}{(h^2 + |\nabla h|^2)^2}  \\  
      &\qquad 
     + \left(2(p-1)(n+1-q) \frac{h_i}{h}+ 2(n+1-q)(\log f)_i\right)\frac{h_i A_{ii}}{h^2 + |\nabla h|^2} \\ 
      &\qquad 
    +(p-1)^2 \frac{h_i^2}{h^2} + 2(p-1)(\log f)_i\frac{h_i}{h} + (\log f)_i^2 ,
    \end{split}
\end{equation}
we also have
\begin{equation}\label{d11}
\begin{split}
    \tilde{f}_{ii} &= (p-1)\frac{h_{ii}}{h} - (p-1)\frac{h^2_i}{h^2} + (\log f)_{ii} \\ 
    &\qquad 
    +(n+1-q)\frac{h_i^2 + h h_{ii} + h_{ii}^2+ \sum_{k}h_k h_{kii}}{h^2 + |\nabla h|^2} - 2(n+1-q)\frac{h_i^2 A_{ii}^2}{(h^2 + |\nabla h|^2)^2}.
\end{split}
\end{equation}
Because the text function $P$ attains its maximum value in $\xi_0\in C_{\theta}$, we have
\begin{equation*}
(\sigma_1)_i =  - h_i h_{ii} - Mh_i,
\end{equation*}
and 
\begin{equation}\label{d14}
    \begin{split}
        0&\geq P_{ij} = (\sigma_1)_{ij} + h_{ii}h_{jj} + h_k h_{kij} + Mh_{ij} \\
        &=
         (\sigma_1)_{ij} + h_{ii}h_{jj} + A_{ijk}h_k - h_i h_j + M h_{ij}\\
    \end{split}
\end{equation}
for the standard metric on spherical cap $C_{\theta}$, using the commutator formula $h_{kij} = h_{ijk} + h_k \delta_{ij} - h_j \delta_{ki}$, then we have
\begin{equation}\label{d12}
\begin{split}
\tilde{f}_{ii} &= (p-1)\frac{h_{ii}}{h} - (p-1)\frac{h_i^2}{h^2} + (\log f)_{ii} \\ 
&\qquad 
+(n+1-q)\frac{h_{ii}^2 + hh_{ii} +  \sum_{k}h_k A_{iik} }{h^2 + |\nabla h|^2} - 2(n+1-q)\frac{h_i^2 A_{ii}^2}{(h^2 + |\nabla h|^2)^2}\\
&= 
(n+1-q)\frac{A_{ii}^2}{h^2 + |\nabla h|^2} - 2(n+1-q)\frac{h_i^2 A_{ii}^2}{(h^2 + |\nabla h|^2)^2}+ (n+1-q)\frac{\sum_k h_k A_{iik}}{h^2 + |\nabla h|^2} \\ 
&\qquad 
-(n+1-q)\frac{hA_{ii}}{h^2 + |\nabla h|^2} + (p-1)\frac{A_{ii}}{h} - (p-1)\frac{h_i^2}{h^2} - (p-1) + (\log f)_{ii}.
\end{split}
\end{equation}
So we obtain that
\begin{equation}
    \begin{split}
        \Delta \tilde{f} &= 
        \frac{n+1-q}{h^2 + |\nabla h|^2}\sum_i \left(  1-\frac{2h_i^2}{h^2 + |\nabla h|^2} \right)A_{ii}^2 - (n+1-q) \frac{\sum_i (h_i^2+ h) A_{ii}}{h^2 + |\nabla h|^2}+ \frac{(p-1)\sigma_1}{h}\\
        &\qquad 
        + (n+1-q)(h-M)\frac{|\nabla h|^2}{h^2 + |\nabla h|^2}-(p-1)\frac{|\nabla h|^2}{h^2} - (p-1) + \Delta(\log f).\\
    \end{split}
\end{equation}
From the arithmetric-geometric mean inequality, we have
\begin{equation}\label{d15}
    \sum_{i=1}^n F^{ii} \geq n(\det A)^{-\frac{1}{n}}.
\end{equation}
Combining \eqref{d6}, \eqref{d7}, \eqref{d14} and \eqref{d15}, then we have
\begin{equation}\label{d13}
\begin{split}
    0&\geq F^{ij}P_{ij} = 
     F^{ij}(\sigma_1)_{ij}+ \sigma_1 + h^2\sum_i F^{ii} - 2nh \\
     &\qquad 
     +\sum_i \tilde{f}_i h_i - F^{ii}h_i^2 + nM - Mh\sum_i F^{ii} \\
     &\geq 
     +\frac{1}{n} \sum_i (\tilde{f}_i)^2+\Delta\tilde{f} + \sigma_1  \sum_i F^{ii}+ \sigma_1 + (h^2-Mh) \sum_i F^{ii} - \sum_i F^{ii}h_i^2\\
      &\qquad
      +\sum_i \tilde{f}_i h_i -2nh+ nM- n\\
      &\geq
      \frac{n+1-q}{(h^2 + |\nabla h|^2)^2}\sum_i \left(  \frac{n+1-q}{n}h_i^2 + h^2 + |\nabla h|^2 -2h_i^2 \right)A_{ii}^2\\
      &\qquad 
       +\frac{1}{n}\left( 2(p-1)(n+1-q)\frac{h_i}{h} + 2(n+1-q)(\log f)_i \right)\frac{h_i A_{ii}}{h^2 + |\nabla h|^2} \\
      &\qquad 
      -(n+1-q) \frac{ h}{h^2 + |\nabla h|^2}\sigma_1+ \sigma_1 \sum_i F^{ii} + \sigma_1+ \frac{p-1}{h}\sigma_1\\
      &\qquad 
       +\sum_i(h^2 - h_i^2-Mh) F^{ii} + (n+1-q)(h-M)\frac{|\nabla h|^2}{h^2 + |\nabla h|^2}\\
       &\qquad 
       + \left(\frac{(p-1)^2}{n}-(p-1) + (p-1)h\right)\frac{|\nabla h|^2}{h^2} + \left(\frac{2(p-1)}{nh} +1\right)\sum_i h_i (\log f)_i\\
       &\qquad 
       + \frac{1}{n}\sum_i(\log f)_i^2-2nh+ \Delta(\log f)-(p-1)+n(M-1).
\end{split}
\end{equation}

Next, we divide it into two subcases: $q=n+1$ and $q\leq 1$.

{\bf Subcase 1.1}. When $q=n+1$, choosing $M:=-\frac{(1+\cot^2\theta)\Vert h\Vert^2_{C^0 (C_{\theta})}}{\min_{C_{\theta}}h}$, \eqref{d13} degenerates into
\begin{equation}\label{d16}
    \begin{split}
        0&\geq \frac{n}{h} \sigma_1 +  \left( h^2 - |\nabla h|^2 - Mh \right)\sum_i F^{ii} -\left( \frac{2(p-1)}{nh} + 1 \right)|\nabla h| |\nabla (\log f)| \\ 
        &\qquad 
         - 2nh  + \Delta(\log f) - (p-1) + n(M-1)\\
         &\geq 
        \frac{n}{h} \sigma_1  -\left( \frac{2(p-1)}{nh} + 1 \right)|\nabla h| |\nabla (\log f)|  - 2nh  + \Delta(\log f) - (p-1) + n(M-1).
    \end{split}
\end{equation}
Together with Lemma \ref{C0 estimate} and Lemma \ref{c1-1}, we conclude that 
$$\sigma_1 (\xi_0) \leq C,$$
where the positive constant $C$ depends on $n$, $p$, $\min_{C_{\theta}}f$, $\min_{C_{\theta}}h$, $\Vert f\Vert_{C^2 (C_{\theta})} $ and $\Vert h\Vert_{C^0 (C_{\theta})}$.

{\bf Subcase 1.2}. When $q\leq 1$, also choosing $M:=-\frac{(1+\cot^2\theta)\Vert h\Vert^2_{C^0 (C_{\theta})}}{\min_{C_{\theta}}h}$, \eqref{d13} degenerates into 
\begin{equation}\label{d17}
    \begin{split}
        0&\geq 
        \frac{h^2 }{(h^2 + |\nabla h|^2)^2}\sigma_1^2 -c_1 \sigma_1 + (h^2 - |\nabla h|^2-Mh)\sum_i F^{ii}- c_2 \\ 
        &
        \geq \frac{h^2 }{(h^2 + |\nabla h|^2)^2}\sigma_1^2 -c_1 \sigma_1 - c_2 ,
    \end{split}
\end{equation}
where the positive constant $c_1$ depends on $n$, $p$, $q$, $\min_{C_{\theta}}f$, $\min_{C_{\theta}}h$, $\Vert f\Vert_{C^1(C_{\theta})}$ and $\Vert h\Vert_{C^0 (C_{\theta})}$, the positive constant $c_2$ depends on $n$, $p$, $q$, $\min_{C_{\theta}}f$, $\min_{C_{\theta}}h$, $\Vert f\Vert_{C^2(C_{\theta})}$ and $\Vert h\Vert_{C^0 (C_{\theta})}$. So, we conclude that 
$$\sigma_1 (\xi_0)\leq C,$$
where the positive constant $C$ depends on $n$, $p$, $q$, $\min_{C_{\theta}}f$, $\min_{C_{\theta}}h$, $\Vert f\Vert_{C^2 (C_{\theta})} $ and $\Vert h\Vert_{C^0 (C_{\theta})}$.

{\bf Case 2.} $\xi_0 \in \partial C_{\theta}$. We choose an orthonormal frame $\{e_i\}_{i=1}^n$ around $\xi_0\in \partial C_{\theta}$ satisfying $e_n = \mu$ at $\xi_0$ and $A_{11}\leq A_{22}\leq \cdots \leq A_{(n-1)(n-1)}$. Then at $\xi_0$ we have
\begin{equation}\label{d27}
    \begin{split}
        0\leq F^{nn}P_n &=(n+1-q)\cot\theta\frac{h}{h^2 + |\nabla h|^2}A_{nn}+\cot\theta \sum_i \left(F^{nn}- F^{ii} \right)\left( A_{nn}-A_{ii} \right)  \\
        &\quad 
       - \cot\theta h^2 F^{nn} + M\cot\theta h F^{nn} +\cot\theta h+ (p-1)\cot\theta  + (\log f)_n \\
       &=
       (n+1-q)\cot\theta\frac{h}{h^2 + |\nabla h|^2}A_{nn}+ \cot\theta \left( 2n - F^{nn}\sigma_1 -A_{nn}\sum_i F^{ii}\right) \\
       &\quad 
              - \cot\theta h^2 F^{nn} + M\cot\theta h F^{nn} +\cot\theta h+ (p-1)\cot\theta  + (\log f)_n\\
    &\leq 
    (n+1-q)\cot\theta\frac{h}{h^2 + |\nabla h|^2}A_{nn}+ \cot\theta \left( 2n - F^{nn}\sigma_1 -A_{nn}\sum_i F^{ii}\right)\\
    &\quad 
    +\cot\theta h + |p-1|\cot\theta + |\nabla_n (\log f)|.\\
    \end{split}
\end{equation}
We assume that $A_{nn}$ large enough, otherwise, the conclusion follows. We know that 
\begin{equation*}
    A_{11}^{n-1}\leq \frac{\det A}{A_{nn}}\leq \frac{c_3}{A_{nn}},
\end{equation*}
where the positive constant $c_3$ depends on $n$, $p$, $q$, $\Vert f\Vert_{C^0(C_{\theta})}$ and $\Vert h\Vert_{C^0 (C_{\theta})}$.
So 
\begin{equation}\label{d28}
    A_{11}\leq \frac{c_4}{A_{nn}^{\frac{1}{n-1}}},
\end{equation}
where the positive constant $c_4$ depends on $n$, $p$, $q$, $\Vert f\Vert_{C^0(C_{\theta})}$ and $\Vert h\Vert_{C^0 (C_{\theta})}$.
Therefore, 
\begin{equation}\label{d29}
    \sum_i F^{ii} \geq \frac{1}{A_{11}} \geq c_4^{-1} A_{nn}^{\frac{1}{n-1}}.
\end{equation}
Inserting \eqref{d29} into \eqref{d27}, we get 
\begin{equation}\label{d30}
    \begin{split}
        0&\leq 
      \cot\theta\left\{  (n+1-q) \frac{ h}{h^2 + |\nabla h|^2}- c_4^{-1}A_{nn}^{\frac{1}{n-1}} \right\}A_{nn} -\cot\theta F^{nn}\sigma_1\\
      &\quad 
      + 2n\cot\theta  +h\cot\theta + |p-1|\cot\theta + |\nabla_n (\log f)| \\ 
      &\leq 
      -\cot\theta F^{nn}\sigma_1 + c_5,\\
    \end{split}
\end{equation}
 where the positive constant $c_5$ depends on $n$, $p$, $\min_{C_{\theta}}f$, $\min_{C_{\theta}}h$, $\Vert f\Vert_{C^1(C_{\theta})}$ and $\Vert h\Vert_{C^0 (C_{\theta})}$. Therefore, we have
\begin{equation}\label{d31}
        \sigma_1(\xi_0) \leq  \frac{c_5}{\cot\theta}A_{nn}.
\end{equation}
Next, we show that $A_{nn}$ cannot be large enough. Combining \eqref{d27} and \eqref{d29}, we have 
\begin{equation}\label{d32}
    0\leq (n+1-q)\cot\theta \frac{h}{h^2 + |\nabla h|^2}A_{nn} - c_4^{-1}\cot\theta A_{nn}^{1+\frac{1}{n-1}} + c_5,
\end{equation}
 then we have
 \begin{equation}\label{d33}
     A_{nn}\leq c_6,
 \end{equation}
 where the positive constant $c_6$ depends on $n$, $p$, $q$, $\min_{C_{\theta}}h$, $\Vert f\Vert_{C^0(C_{\theta})}$ and $\Vert h\Vert_{C^0 (C_{\theta})}$.
Hence, $A_{nn}$ cannot be large enough and $\sigma_1$ is bounded above.

\end{proof}

Based on the above priori estimates, then we have
\begin{theorem}\label{c2 theorem}
    Let $p>q$, $q\leq 1$ and $\theta\in (0,\frac{\pi}{2})$. Suppose that $h$ is a solution to Eq. \eqref{Eq}, then there holds
    \begin{equation}\label{g1}
        \min_{C_{\theta}} h \geq c,
    \end{equation}
    and for any $\alpha\in (0,1)$,
    \begin{equation}\label{g2}
        \Vert h\Vert_{C^{3,\alpha}(C_{\theta})}\leq C,
    \end{equation}
    where the positive constant $c$, $C$ depend only on $n$, $p$, $q$ and $f$. 
\end{theorem}

\begin{proof}
    Combining Lemma \ref{C0 estimate}, Lemma \ref{c1-1} and Lemma \ref{c2 estimate}, we obtain
    \begin{equation*}
        c\leq \min_{C_{\theta}}h,\qquad and \qquad \Vert h\Vert_{C^2(C_{\theta})}\leq C.
    \end{equation*}
    By the theory of fully nonlinear second-order uniformly elliptic equations and the Schauder estimate we have \eqref{g2}.
\end{proof}

\section{Proof of Theorem 1.1} \label{se5}
In this section, we use the continuity method to complete the proof of Theorem 1.1. Let 
\begin{equation*}
    f_t : = (1-t)l^{1-p}\left(l^2 + |\nabla l|^2\right)^{\frac{q-n-1}{2}} + tf, \qquad for \quad 0\leq t\leq 1.
\end{equation*}
Consider the following family of equations
    \begin{equation}\label{h1}
\left\{
\begin{aligned}
&\det\left(\nabla^2 h  + h\sigma\right) = h^{p-1}\left(h^2 + |\nabla h|^2\right)^{\frac{n+1-q}{2}}f_t, ~~&&in~C_{\theta}\\
&\nabla_{\mu} h=\cot\theta h,  ~~&& on~\partial C_{\theta}
\end{aligned}
\right.
\end{equation}
Define the set
\begin{equation*}
    \mathcal{H} :=\{ h\in C^{4,\alpha}(C_{\theta}) | \nabla_{\mu}h = \cot\theta h~~on~~\partial C_{\theta}\},
\end{equation*}
and
\begin{equation*}
    \mathcal{I} := \{t\in [0,1] | \rm{Eq. ~\eqref{h1}~has~a~positive~solution~in~\mathcal{H}}\}.
\end{equation*}
For notational simplicity, define
\begin{equation}\label{h2}
    \mathcal{G}(h_{ij},h) := \det(A),
\end{equation}
and 
\begin{equation}\label{h3}
    \mathcal{J}_t(\nabla h, h) : =h^{p-1}(h^2 + |\nabla h|^2)^{\frac{n+1-q}{2}}f_t. 
\end{equation}
We rewrite Eq. \eqref{Eq} as 
    \begin{equation}\label{REq}
\left\{
\begin{aligned}
&\mathcal{G}(h_{ij},h) = \mathcal{J}_1 (\nabla h,h), ~~&&in~C_{\theta}\\
&\nabla_{\mu} h=\cot\theta h,  ~~&& on~\partial C_{\theta}
\end{aligned}
\right.
\end{equation}

In order to prove Theorem \ref{main1.1}, we first show the uniqueness part. 

\begin{lemma}\label{uniqueness}
    Let $p>q$ and $\theta\in (0,\frac{\pi}{2})$, suppose $h^1$, $h^2\in C^{2,\alpha}(C_{\theta})$ are solutions to the Eq. \eqref{REq}. Then $h^1\equiv h^2$.
\end{lemma}

\begin{proof}
    We prove this by contradiction. Without loss of generality, we may assume $h^1>h^2$ somewhere on $C_{\theta}$. Hence, there exists a constant $m\geq 1$ such that
    \begin{equation*}
        mh^2 - h^1 \geq 0 \quad on\quad C_{\theta},\quad and \quad mh^2 - h^1 = 0 \quad \rm{at~some~point~} P\in C_{\theta}.
    \end{equation*}
For $p>q$, by homogeneity of $\mathcal{G}$ and $\mathcal{J}_1$ and $m\geq 1$, we have
\begin{equation*}
    \mathcal{G}(h_{ij}^1, h^1) = \mathcal{J}_1(\nabla h^1,h^1),
\end{equation*}
and
\begin{equation*}
    \mathcal{G}(mh_{ij}^2, mh^2) = m^{q-p}\mathcal{J}_1(\nabla mh^2,mh^2) \leq \mathcal{J}_1 (\nabla mh^2, mh^2).
\end{equation*}
Hence
\begin{equation*}
    \begin{split}
        0&\geq 
        \mathcal{G}(mh^2_{ij}, mh^2) - \mathcal{G}(h^1_{ij}, h^1) + \mathcal{J}_1 (\nabla h^1, h^1) - \mathcal{J}_1 (\nabla mh^2, mh^2) \\ 
        &= 
        \int_{0}^1 \frac{d\mathcal{G}(\varepsilon m h^2_{ij} + (1-\varepsilon)h^1_{ij}, \varepsilon mh^2+(1-\varepsilon)h^1)}{d
        \varepsilon}d\varepsilon - \int_0^1 \frac{d\mathcal{J}_1(\varepsilon \nabla mh^2 + (1-\varepsilon)\nabla h^1)}{d\varepsilon}d\varepsilon \\ 
        &=
        \int_0^1 \frac{\partial\mathcal{G}(\varepsilon mh^2_{ij} + (1-\varepsilon)h^1_{ij}, \varepsilon mh^2 + (1-\varepsilon)h^1)}{\partial h_{ij}}d\varepsilon \left[(mh^2 - h^1)_{ij} + (mh^2 - h^1)\delta_{ij}\right] \\  
        &\qquad 
        -\int_0^1 \frac{\partial\mathcal{J}_1 (\varepsilon\nabla mh^2 + (1-\varepsilon)\nabla h^1)}{\partial h_i}d\varepsilon (mh^2 - h^1)_i \\ 
        &\qquad 
        - \int_0^1 \frac{\partial \mathcal{J}_1 (\varepsilon \nabla mh^2 + (1-\varepsilon)\nabla h^1)}{\partial h}d\varepsilon (mh^2 - h^1).
    \end{split}
\end{equation*}
That is, $mh^2-h^1$ satisfies an elliptic inequality. The strong maximum principle \cite[Theorem 3.5]{GT83} yields $mh^2\equiv h^1$ on $C_{\theta}$. This implies that $mh^2$ solves Eq. \eqref{REq}, i.e.,
\begin{equation*}
    \begin{split}
        \mathcal{G}(mh^2_{ij}, mh^2) & = \mathcal{J}_1 (\nabla mh^2, mh^2) \\
        &= m^{n+p-q}\mathcal{J}_1(\nabla h^2,h^2) \\ 
        &= m^{n+p-q} \mathcal{G}(h^2_{ij}, h^2).
    \end{split}
\end{equation*}
    By the homogeneity of $\mathcal{G}$ implies
    \begin{equation*}
        \mathcal{G}(mh^2_{ij}, mh^2) = m^n \mathcal{G}(h^2_{ij}, h^2).
    \end{equation*}
Thus, we have $m^{p-q} = 1$, which implies that $m=1$.
\end{proof}

We define $\mathcal{L}_h$ to be the linearized operator of Eq. \eqref{REq} at $h$, that is, for each $\gamma\in C^{2,\alpha}(C_{\theta})$,
\begin{equation}\label{h4}
    \begin{split}
        \mathcal{L}_h(\gamma) &=\frac{d(\mathcal{G}((h_{\varepsilon}\gamma)_{ij}, h_{\varepsilon}\gamma) - \mathcal{J}_1(\nabla(h_{\varepsilon}\gamma),h_{\varepsilon}\gamma))}{d\varepsilon}|_{\varepsilon=0} \\ 
        &=
        \frac{\partial\mathcal{G}}{\partial h_{ij}}((h\gamma)_{ij}+h\gamma\delta_{ij}) - \frac{\partial\mathcal{J}_1}{\partial h_i}(h\gamma)_i - \frac{\partial\mathcal{J}_1}{\partial h}h\gamma,
    \end{split}
\end{equation}
where $h_{\varepsilon} = he^{\varepsilon\gamma}$.

\begin{lemma}\label{kernel}
    Let $p>q$ and $\theta\in (0,\frac{\pi}{2})$. Suppose $\gamma\in\mathcal{H}$ and satisfies
    $$\mathcal{L}_h (\gamma) = 0,$$
    then $\gamma\equiv 0 $ on $C_{\theta}$.
\end{lemma}

\begin{proof}
    From $\mathcal{L}_h (\gamma) = 0$ and \eqref{h4}, we have
    \begin{equation*}
        \begin{split}
            0= \mathcal{L}_h(\gamma) &= \frac{\partial \mathcal{G}}{\partial h_{ij}}((h\gamma)_{ij} + h\gamma\delta_{ij}) - \frac{\partial\mathcal{J}_1}{\partial h_i}(h\gamma)_i - \frac{\partial\mathcal{J}_1}{\partial h}h\gamma \\ 
            &= 
            \gamma\mathcal{L}_h (1) + h\frac{\partial\mathcal{G}}{\partial h_{ij}}\gamma_{ij} + 2\frac{\partial\mathcal{G}}{\partial h_{ij}}h_i \gamma_j - h\frac{\partial\mathcal{J}_1}{\partial h_i}\gamma_i.
        \end{split}
    \end{equation*}
Note that the matrix $(\frac{\partial\mathcal{G}}{\partial h_{ij}})_{n\times n}$ is positive definite. Therefore, at any minimum point of $\gamma$, one has
$$h\frac{\partial\mathcal{G}}{\partial h_{ij}}\gamma_{ij} + 2\frac{\partial\mathcal{G}}{\partial h_{ij}}h_i\gamma_j - h\frac{\partial\mathcal{J}_1}{\partial h_i}\gamma_i\geq 0,$$
which implies 
$$\gamma\mathcal{L}_h (1)\leq 0.$$
By homogeneity of the Monge-Amp\`ere equation
$$\mathcal{G}(h_{ij},h) = \mathcal{J}_1(\nabla h, h),$$
    it's simple to see that if $p>q$,
    $$\mathcal{L}_h(1) = (q-p)\mathcal{J}_1 (h,\nabla h)<0.$$
    That means, the minimum value of $\gamma$ must be nonnegative. Similarly, the maximum value of $\gamma$ must be nonpositive. Hence $\gamma\equiv 0$.
\end{proof}

Now, we give the prove of Theorem \ref{main1.1}.

\begin{proof}
    Clearly $0\in\mathcal{I}$ since $h=l\in\mathcal{H}$ is a solution to $\eqref{h1}$ when $t=0$, so the set $\mathcal{I}$ is nonempty. We prove that $\mathcal{I} = [0,1]$ by showing $\mathcal{I}$ is both open and closed. The closeness of $\mathcal{I}$ follows from Theorem \ref{c2 theorem}, while the openness of $\mathcal{I}$ follows from Lemma \ref{kernel}. Then it follows that $1\in\mathcal{I}$, i.e., we obtain a positive solution to Eq. \eqref{REq}. This completes the proof.
\end{proof}

\vspace {5 mm}

\section*{Acknowledgments}

This research was supported in part by the China Postdoctoral Science Foundation under Grant Number 2025M773056, the Beijing Postdoctoral Research Foundation.

\end{document}